\documentclass[a4paper,reqno,oneside,11pt]{amsart}

\usepackage[utf8]{inputenc}

\usepackage[dvipsnames]{xcolor} % Must come before tikz!
\colorlet{cite}{red}
\usepackage{tikz}  % for diagrams

\usetikzlibrary{arrows,positioning}
\tikzset{ 
  baseline=-2.3pt,
  text height=1.5ex, text depth=0.25ex,
  >=stealth,
  node distance=2cm,
  mid/.style={fill=white,inner sep=2.5pt},
}

\usepackage{lmodern}
\usepackage{tikz-cd}
\usepackage{amsthm, amssymb, amsfonts}

\usepackage{amsthm, amssymb, amsfonts}	% Standard maths packages.
\usepackage[all]{xy}
\usepackage{graphicx,caption,subcaption}
\usepackage{microtype}
\usepackage{DotArrow}
\usepackage[makeroom]{cancel}
\usepackage[%
  bookmarks=true,			% show bookmarks bar?
  unicode=true,			% non-Latin characters in Acrobat’s bookmarks
  pdftitle={F-Lie groups},		%
  pdfauthor={Torres-Gomez, Valencia},	% 
  pdfkeywords={Pseudo-Riemannian, F-manifold, Lie group},	% list of keywords
  colorlinks=true,		% false: boxed links; true: colored links
  linkcolor=black,			% color of internal links
  citecolor=black,		% color of links to bibliography
  filecolor=magenta,		% color of file links
  urlcolor=RoyalBlue			% color of external links
]{hyperref}				% One of the most useful packages I have found.
\newtheoremstyle{mydef}
  {}		% Space above environment
  {}		% Space below environment
  {}		% Body font
  {}		% Indent amount (empty = no indent, \parindent = para indent)
  {\scshape}	% theorem head font
  {. }		% Punctuation after heading
  { }		% Space after heading
  {\thmname{#1}\thmnumber{ #2}\thmnote{ #3}}	% Heading spec

\newtheorem{theorem}{Theorem}[section]
\newtheorem*{theorem*}{Theorem}
\newtheorem{proposition}[theorem]{Proposition}
\newtheorem*{proposition*}{Proposition}
\newtheorem{lemma}[theorem]{Lemma}
\newtheorem*{lemma*}{Lemma}
\newtheorem{corollary}[theorem]{Corollary}
\newtheorem*{corollary*}{Corollary}
\theoremstyle{definition}
\newtheorem{definition}[theorem]{Definition}
\newtheorem{example}[theorem]{Example}

\theoremstyle{remark}
\newtheorem{remark}[theorem]{Remark}

\newtheorem*{conjecture*}{Conjecture}
\usepackage{tikz}

\author{Alexander Torres-Gomez {\tiny and} Fabricio Valencia}
\address{}
\date{\today}

\address{A. Torres-Gomez  - Instituto de Matem\'aticas, Universidad de Antioquia , Calle 70 $\#$ 52–21, Medell\'in - Colombia.
	\newline  
      \phantom{xx}
  F. Valencia - Instituto de Matem\'atica e Estat\'istica, Universidade de S\~ao Paulo, Rua do Mat\~ao 1010, Cidade Universit\'aria, 05508-090 S\~ao Paulo - Brasil.
  \newline
	\phantom{xx}
galexander.torres@udea.edu.co, fabricio.valencia@ime.usp.br}

\title{Double extension of flat pseudo-Riemannian $F$-Lie algebras}

\begin{document}
\maketitle

\begin{abstract}
We define the concept of a flat pseudo-Riemannian $F$-Lie algebra and construct its corresponding double extension. This algebraic structure can be interpreted as the infinitesimal analogue of a Frobenius Lie group devoid of Euler vector fields. We show that the double extension provides a framework for generating all weakly flat Lorentzian non-abelian bi-nilpotent $F$-Lie algebras possessing one-dimensional light-cone subspaces.  A similar result can be established for nilpotent Lie algebras equipped with flat scalar products of signature $(2,n-2)$ where  $n\geq 4$. Furthermore, we use this technique to construct Poisson algebras exhibiting compatibility with flat scalar products. 
\end{abstract}

%\tableofcontents
\section{Introduction}

Frobenius manifolds emerged as a geometric incarnation of the Witten-Dijkgraaf-Verlinde-Verlinde equation within the framework of two-dimensional topological field theories. Dubrovin \cite{Dubrovin} pioneered the study of these geometric objects. Subsequently, Hertling and Manin \cite{HertlingManin} introduced the notion of an $F$-manifold, a geometric construction that relaxes some conditions of a Frobenius manifold and serves as the primary motivation for our current investigation. Notably, all Frobenius manifolds are inherently $F$-manifolds.\\

The ubiquity of these manifolds extends across various mathematical and physical contexts, leading to a wealth of intriguing applications and consequences. A comprehensive discussion, along with relevant references, can be found in \cite{Hertling}. Specifically, $F$-manifolds exhibit close connections to singularity theory, quantum cohomology, symplectic geometry, and the realm of integrable systems.\\

Our investigation centers on $F$-manifolds satisfying an additional weak Frobenius-like identity. We use algebraic techniques to achieve the twofold objective of constructing and classifying invariant examples of these geometric objects.  Specifically, this work introduces a double extension process tailored to constructing and exploring \emph{flat pseudo-Riemannian F-Lie algebras} (as defined in Definition \ref{mainDef3} below). This process allows us to explore some of their inherent properties. We build upon the double extension process developed in \cite{AubertMedina} for flat pseudo-Riemannian Lie algebras. Additionally, we leverage concepts from Nijenhuis cohomology for left-symmetric algebras \cite{Nijenhuis} and Hochschild cohomology for associative algebras  \cite{Hochschild, Harrison}, with a focus on the associative-commutative case.\\

The paper is organized in the following manner. 
\begin{itemize}

\item Section \ref{S:FPRFS} introduces the concept of a flat pseudo-Riemannian $F$-structure on a 
smooth manifold, with a specific focus on Lie groups and their Lie algebras. It also discusses extensions of this structure and its relation to Frobenius manifolds.\\

\item Section \ref{S:DE} presents a method for constructing $(n + 2)$-dimensional (weakly) flat pseudo-Riemannian $F$-Lie algebras from $n$-dimensional (weakly) ones. This method involves an explicit set of parameters (structure coefficients) satisfying specific algebraic equations. As a key result, we obtain all weakly flat Lorentzian non-abelian bi-nilpotent $F$-Lie algebras with 1-dimensional light-cone subspaces that also function as associative two-sided ideals. Additionally, we prove the non-existence of flat Lorentzian non-abelian bi-nilpotent $F$-Lie algebras $(\mathfrak{g},\circ, \langle \cdot,\cdot\rangle)$ with the following properties: a 1-dimensional two-sided ideal $\mathbb{R}\hat{a}$ of $(\mathfrak{g},\circ)$ within the intersection of the kernel $N(\mathfrak{g})$ and the center $Z(\mathfrak{g})$ with $\langle \hat{a},\hat{a}\rangle=0$. This section also extends these results to nilpotent Lie algebras equipped with flat scalar products of signature $(2,n-2)$ for $n\geq 4$.\\

\item Section \ref{S:TSC} describes how the core construction can be adapted to generate two additional double extension processes applicable to Lie algebras of $F$-strong symmetric type and Frobenius-like Poisson algebras.\\

\item Section \ref{S:E} provides concrete examples by applying the main constructions developed throughout the paper.\\

\end{itemize}

\vspace{0.3cm}
{\bf Acknowledgments:} F. Valencia was supported by Grant 2020/07704-7 Sao Paulo Research Foundation - FAPESP.

\section{Flat Pseudo-Riemannian $F$-structures}\label{S:FPRFS}
This section establishes the key definitions that will be employed throughout the remainder of this paper.\\

\begin{definition}\label{mainDef}
A \emph{flat pseudo-Riemannian $F$-structure} on a smooth manifold $M$ is a triple $(\circ, \eta,e)$, where $\circ:\mathfrak{X}(M)\times \mathfrak{X}(M)\to \mathfrak{X}(M)$ is an associative and commutative $C^\infty(M)$-bilinear multiplication, $e$ is a vector field on $M$ that acts as a unit element with respect to $\circ$, and $\eta$ is a flat pseudo-Riemannian metric on $M$ such that:
\begin{enumerate}
\item[i.] $\eta(X\circ Y,Z)=\eta(X,Y\circ Z)$ for all $X,Y,Z\in \mathfrak{X}(M)$, 
\item[ii.] $e$ is flat with respect to the Levi--Civita connection $\nabla$ associated to $\eta$, i.e. $\nabla e=0$, and
\item[iii.] the following Hertling-Manin relation holds
\begin{equation}\label{Hertling-ManinEquation0}
	\textnormal{Lie}_{X\circ Y}(\circ)=X\circ \textnormal{Lie}_Y(\circ)+Y\circ \textnormal{Lie}_X(\circ),\quad X,Y\in  \mathfrak{X}(M).
\end{equation}
Here, ``$\textnormal{Lie}$'' denotes the Lie derivative of tensor fields. 
\end{enumerate}

A smooth manifold $M$ equipped with a flat pseudo-Riemannian $F$-structure $(\circ, \eta,e)$ is referred to as a \emph{flat pseudo-Riemannian $F$-manifold}. We can relax some conditions of a flat pseudo-Riemannian $F$-structure to obtain a weaker geometric object. Specifically, $(M,\circ,\eta)$ is called a \emph{weakly flat pseudo-Riemannian $F$-manifold} if it satisfies all the conditions of a flat pseudo-Riemannian $F$-structure except possibly one of the following: the flatness condition $\nabla e=0$, and/or the identity element condition $e\circ=\textnormal{id}_{\mathfrak{X}(M)}$.
\end{definition}

We establish some crucial properties of flat pseudo-Riemannian $F$-manifolds that will be instrumental throughout this work.

\begin{itemize}
\item Frobenius Algebra Structure: for each point $p\in M$, the tangent space $T_pM$ equipped with the product $\circ_p$ and the metric $\eta_p$ becomes a Frobenius algebra. The unit vector field $e$ satifies $\textnormal{Lie}_{e}(\circ)=0$. Furthermore, assuming $e$ is flat, the 1-form $\alpha_e$ defined by $\alpha_e(X)=\eta(X,e)$ is closed.\\

\item The Hertling-Manin Tensor: the Hertling-Manin identity \eqref{Hertling-ManinEquation0} can be expressed in a more explicit form as
\begin{eqnarray*}
	&  & [X\circ Y,Z\circ W]-[X\circ Y,Z]\circ W-[X\circ Y,W]\circ Z -X\circ [Y,Z\circ W]+X\circ [Y,Z]\circ W\\
	&  & + X\circ [Y,W]\circ Z-Y\circ [X,Z\circ W]+Y\circ [X,Z]\circ W+ Y\circ [X,W]\circ Z=0,
\end{eqnarray*}
for any four arbitrary (local) vector fields $X,Y,Z,W$ on $M$. The left-hand side of the above expression defines a $(4, 1)$-tensor field denoted by $HM(X, Y, Z, W)$. This tensor exhibits specific symmetries: symmetric in $X, Y$ and $W, Z$, and skew-symmetric when swapping $X, Y$ with $W, Z$. Furthermore, $HM$ can be expressed, using the Leibnizator $\mathcal{L}$, defined by

$$\mathcal{L}(X,Y,Z):=[X,Y\circ Z]-[X,Y]\circ Z-[X,Z]\circ Y \; ,$$
as
$$HM(X,Y,X,W)=\mathcal{L}(X\circ Y,Z,W)-X\circ \mathcal{L}(Y,Z,W)-Y\circ \mathcal{L}(X,Z,W).$$

The Leibnizator is symmetric in its second and third arguments, that is, $\mathcal{L}(X,Y,Z)=\mathcal{L}(X,Z,Y)$.\\

\item Weak Frobenius Manifolds: introduce a $(3, 0)$-tensor field $A$ defined by $A(X,Y,Z)=\eta(X\circ Y, Z)$, for all $X,Y,Z\in \mathfrak{X}(M)$. Based on results from \cite{Hertling} (Theorem 2.15) if the covariant derivative $(\nabla_WA)(X,Y,Z)$ exhibits symmetry in all four arguments, then the Hertling-Manin identity holds. However, the converse is not always true, as our constructions will show. These properties allow us to consider flat pseudo-Riemannian $F$-manifolds as a special case of weak Frobenius manifolds without Euler vector fields (\cite[p. 22, Remarks 2.17]{Hertling}). Importantly, the expression $(\nabla_WA)(X,Y,Z)$ is always symmetric in $X,Y,Z$. Demanding symmetry in all four arguments is equivalent to requiring the $(3, 1)$-tensor field $\nabla\circ$, defined by $\nabla\circ(X,Y,Z)=\nabla_X(Y\circ Z)-(\nabla_XY)\circ Z-Y\circ (\nabla_XZ)$, to be symmetric in all three arguments. Additionally, this all-symmetry condition is further equivalent to the existence of a local potential $\Phi\in C_p^\infty(U)$ around every $p\in U \subset M$,  such that $(XYX)\Phi=A(X,Y,Z)$ for any flat local vector fields $X,Y,Z$.
\end{itemize}

Note that flat pseudo-Riemannian $F$-manifolds are a specific subclass of pseudo-Riemannian $F$-manifolds introduced in \cite{ArsieBuryakLorenzoniRossi}.\\

We now consider connected and simply connected Lie groups. Let $G$ be such a group with Lie algebra $\mathfrak{g}$. For each element $x\in \mathfrak{g}$, we denote by $x^+$ its corresponding left-invariant vector field on $G$. A multiplicative structure $\circ$ on $G$, as defined in Definition \ref{mainDef}, is said to be \emph{left-invariant} if $x^+\circ y^+$ is left-invariant for all $x,y\in \mathfrak{g}$. This translates to requiring that the left multiplication diffeomorphisms $l_g:G \to G$ (where $g \in G$) preserve the multiplicative structure $\circ$. In other words, for all $g \in G$, $(l_g)_\ast \circ=\circ ((l_g)_\ast\times (l_g)_\ast)$. A flat pseudo-Riemannian $F$-Lie group is a quadruple $(G,\circ,\eta,e)$ where $G$ is a connected and simply connected Lie group and $(\circ,\eta,e)$ is a (weakly) flat pseudo-Riemannian $F$-structure on $G$. Crucially, each of the tensors $\circ,\eta, e$ is left-invariant.\\

As alluded to previously, our objective is to investigate specific characteristics of flat pseudo-Riemannian $F$-Lie groups. This investigation aims to yield noteworthy examples and constructions within this geometric framework.  Consistent with this approach, we will leverage the associated infinitesimal counterpart, the Lie algebra of the Lie group.

\begin{definition}\label{mainDef3}
Let $\mathfrak{g}$ be a real finite-dimensional Lie algebra. A flat pseudo-Riemannian $F$-structure on $\mathfrak{g}$ is a triple  $(\circ, \langle \cdot,\cdot \rangle,e),$ where $\circ:\mathfrak{g}\times \mathfrak{g}\to \mathfrak{g}$ is an associative and commutative bilinear product with a unit element $e \in \mathfrak g$ and $\langle \cdot,\cdot \rangle:\mathfrak{g}\times \mathfrak{g}\to\mathbb{R}$ is a flat scalar product\footnote{By scalar product, we mean a non-degenerate  symmetric bilinear form on $\mathfrak{g}$.}, satisfying the following conditions:

		\begin{enumerate}
		\item[i.] $\langle x\circ y,z \rangle=\langle x,y\circ z \rangle$ for all $x,y,x\in \mathfrak{g}$, 
		\item[ii.] $e$ is flat with respect to the Levi--Civita product $\ast$ associated to $\langle \cdot,\cdot \rangle$, i.e. $x\ast e=0$ for all $x\in \mathfrak{g}$, and
		\item[iii.] the following Hertling-Manin relation holds, for all $x,y,z,w\in \mathfrak{g}$,
				\begin{eqnarray*}
			&  & [x\circ y,z\circ w]-[x\circ y,z]\circ w-[x\circ y,w]\circ z -x\circ [y,z\circ w]+x\circ [y,z]\circ w\\
			&  & + x\circ [y,w]\circ z-y\circ [x,z\circ w]+y\circ [x,z]\circ w+ y\circ [x,w]\circ z=0.
		\end{eqnarray*}
		Here, $[\cdot , \cdot ]$ denotes the Lie bracket on  $\mathfrak{g}$.\\
	\end{enumerate} 

A quadruple $(\mathfrak{g},\circ,\langle \cdot,\cdot \rangle,e)$, satisfying these conditions, is referred to as a \emph{flat pseudo-Riemannian F-Lie algebra}.  If everything except $x\ast e=0$ or $e\circ x =x$ (for all $x\in\mathfrak{g}$) is satisfied, then we say that $(\mathfrak{g},\circ,\langle \cdot,\cdot \rangle)$ is a \emph{weakly flat pseudo-Riemannian F-Lie algebra}.\\
\end{definition}

The formula in item i. of Definition \ref{mainDef3} is referred to as the \emph{Frobenius identity}.  The pair $(\mathfrak{g},\langle \cdot,\cdot\rangle)$ is a well-established concept in the literature known as a \emph{flat pseudo-Riemannian Lie algebra}. These infinitesimal (algebraic) objects and their corresponding global counterparts have been extensively studied (see, e.g.,  \cite{AubertMedina,AitBenHaddouBoucettaLebzioui,BoucettaLebzioui1,BoucettaLebzioui2}).  For our purposes, we will primarily focus on applying some of the techniques developed in \cite{AubertMedina}. The Levi-Civita product $\ast$ on $\mathfrak{g}$ induced by the flat metric $\langle \cdot,\cdot \rangle$ is uniquely determined by the formula
$$2\langle x\ast y,z \rangle = \langle [x,y],z \rangle-\langle [y,z],x \rangle+\langle [z,x],y \rangle,$$
for all $ x,y,z\in \mathfrak{g}$. This formula demonstrably satisfies both the Lie bracket identity $[x,y]=x\ast y-y\ast x$ and the metric compatibility condition $\langle x\ast y,z \rangle+\langle y,x\ast z \rangle=0$. Furthermore, the Levi-Civita product is flat, meaning it defines a left-symmetric product on $\mathfrak{g}$. In other words, it verifies the following identity
$$(x\ast y)\ast z-x\ast(y\ast z)=(y\ast x)\ast z-y\ast (x\ast z),$$
for all $x,y,z\in \mathfrak{g}$ (see, e.g., \cite{Burde} for details). We call a pair $(\mathfrak{g},\ast)$ satisfying this property as a \emph{left-symmetric algebra}. It is worth noting that similar concepts to those in Definition \ref{mainDef3}, without a scalar product and exploring compatibility conditions between associative commutative products and left-symmetric products on vector spaces, have been explored in the literature \cite{GozeRemm,LuiShengBai}.  Our focus, however, extends beyond these existing notions.  We aim to investigate a geometric object that exemplifies what could be termed \emph{weak Frobenius Lie groups with no Euler vector fields}. The constructions developed in \cite{GozeRemm,LuiShengBai} have served as a source of inspiration for our study of Poisson algebras that exhibit compatibility with flat scalar products.  This compatibility can be achieved by leveraging some of the algebraic techniques employed in this work.  This is particularly relevant because the Hertling-Manin relation holds trivially for Poisson algebras, as their Leibnizator vanishes identically.  Consequently, the task of constructing Poisson algebras within our framework becomes significantly more streamlined and natural.

\section{double extension}\label{S:DE}

This section introduces the double extension process for flat pseudo-Riemannian $F$-Lie algebras. To facilitate this, we first provide a concise overview of two relevant cohomology theories. We closely follow the results established in \cite{Dzhumadil,Harrison,Hochschild,Nijenhuis,Remm}. It is worth noting that, while all constructions can be generalized to complex vector spaces, we restrict our focus here to real vector spaces for clarity. 

\subsection*{Nijenhuis cohomology for left symmetric algebras}

Let $(\mathfrak{g},\ast)$ be a real left-symmetric algebra and $V$ be a finite-dimensional real vector space. We say $V$ possesses the structure of a $(\mathfrak{g},\ast)$-\emph{bimodule} if there exist bilinear maps $\vartriangleright:\mathfrak{g}\times V\to V$ and $\vartriangleleft:V\times \mathfrak{g}\to V$ satisfying the following compatibility conditions, for all $x,y\in\mathfrak{g}$ and $ v\in V$,

\begin{eqnarray*}
	&  & x\vartriangleright(y\vartriangleright v)-y\vartriangleright(x\vartriangleright v)=(x\ast y-y\ast x)\vartriangleright v\quad\textnormal{and}\label{bimo1}\\
	&  & x\vartriangleright(v\vartriangleleft y)-(x\vartriangleright v)\vartriangleleft y=v\vartriangleleft (x\ast y)-(v\vartriangleleft x)\vartriangleleft y,
\end{eqnarray*}

We define the cochain complex $(C^\bullet(\mathfrak{g},V),\delta_\bullet)$ where $C^p(\mathfrak{g},V)$ denotes the vector space of $p$-linear maps $f:\mathfrak{g}\times\cdots\times\mathfrak{g}\to V$ for $p\in \mathbb{Z}^+$ ,  $C^0(\mathfrak{g},V):=V$, and $\delta_p:C^p(\mathfrak{g},V)\to C^{p+1}(\mathfrak{g},V)$ is the \emph{Nijenhuis differential}
$$(\delta_p f)(x_0,\cdots,x_p) = \sum_{i=0}^{p-1}(-1)^ix_i\vartriangleright f(x_0,\cdots,\hat{x_i},\cdots,x_p)+\sum_{i=0}^{p-1}(-1)^if(x_0,\cdots,\hat{x_i},\cdots,x_{p-1},x_i)\vartriangleleft x_p$$
$$-\sum_{i<j<p}(-1)^{i+j+1} f(x_i\ast x_j-x_j\ast x_i,\cdots,\hat{x_i},\cdots,\hat{x_j},\cdots,x_p)-\sum_{i=0}^{p-1}(-1)^if(x_0,\cdots,\hat{x_i},\cdots,x_{p-1},x_i\ast x_p).$$\\

It follows that $\delta_p\circ\delta_{p-1}=0$, establishing the structure of a complex.\\

As is standard, we define the spaces of \emph{k-cocycles} and \emph{k-coboundaries} as $Z_{SG}^k(\mathfrak{g},V):=\textnormal{Ker}(\delta_k)$ and $B_{SG}^k(\mathfrak{g},V):=\textnormal{Im}(\delta_{k-1})$ (with $B_{SG}^0(\mathfrak{g},V):=\{ 0\}$), respectively. Therefore, the $k$-th \emph{Nijenhuis cohomology space} of the left-symmetric algebra $(\mathfrak g, \ast)$ with coefficients in $V$ is the quotient space

$$H_{SG}^k(\mathfrak{g},V):=Z_{SG}^k(\mathfrak{g},V)/B_{SG}^k(\mathfrak{g},V).$$

In the specific case where  $V=\mathbb{R}$ is endowed with the trivial $(\mathfrak{g},\ast)$-bimodule structure (i.e. $x\vartriangleright t=t\vartriangleleft x=0$), we referred to the $k$-th Nijenhuis cohomology space $H_{SG}^k(\mathfrak{g},\mathbb{R})$ as the $k$-th \emph{space of scalar cohomology} of the left-symmetric algebra $(\mathfrak{g},\ast)$. The reader is recommended to visit \cite{Dzhumadil,Nijenhuis,Remm} for details.

\subsection*{Hochschild cohomology for associative algebras}

Let $(\mathfrak{g},\circ)$ be a real associative algebra and $V$ be a finite-dimensional real vector space. A $(\mathfrak{g},\circ)$-\emph{bimodule} structure on $V$ is defined by two bilinear maps $\vartriangleright:\mathfrak{g}\times V\to V$ and $\vartriangleleft:V\times \mathfrak{g}\to V$ satisfying the following compatibility conditions, for all $x,y\in\mathfrak{g}$ and $v\in V$,
\begin{equation*}
(x\circ y)\vartriangleright v=x\vartriangleright (y\vartriangleright v),\quad (x \vartriangleright v)\vartriangleleft y=x\vartriangleright(v\vartriangleleft y),\quad\textnormal{and}\quad (v\vartriangleleft x)\vartriangleleft y=v\vartriangleleft (x\circ y) \, .
\end{equation*}

We define the cochain complex $(C^\bullet(\mathfrak{g},V),\delta_\bullet)$ where the \emph{Hochschild differential} $\delta_p:C^p(\mathfrak{g},V)\to C^{p+1}(\mathfrak{g},V)$ is defined as \cite{Hochschild}
\begin{eqnarray*}
(\delta_p f)(x_0,\cdots,x_p) & = & x_0 \vartriangleright f(x_1, \cdots,x_p)+\sum_{i=0}^{p-1}(-1)^{i+1}f(x_0,\cdots, x_{i-1},x_i\circ x_{i+1},x_{i+2},\cdots,x_p)\\
&+&(-1)^{p+1}f(x_0,\cdots,x_{p-1})\vartriangleleft x_p.
\end{eqnarray*}

As in the case of Nijenhuis cohomology, we define the spaces of $k$-cocycles $Z_{AS}^k(\mathfrak{g},V):=\textnormal{Ker}(\delta_k)$ and $k$-coboundaries $B_{AS}^k(\mathfrak{g},V):=\textnormal{Im}(\delta_{k-1})$ (with $B_{AS}^0(\mathfrak{g},V)=\{ 0\}$ by convention). Thus, the $k$-th \emph{Hochschild cohomology space} of the associative algebra $(\mathfrak{g},\circ)$ with coefficients in V is the quotient space
$$H_{AS}^k(\mathfrak{g},V):=Z_{AS}^k(\mathfrak{g},V)/B_{AS}^k(\mathfrak{g},V).$$

In the special case where $V=\mathbb{R}$ is endowed with the trivial $(\mathfrak{g},\circ)$-bimodule structure, we denote the $k$-th Hochschild cohomology space $H_{AS}^k(\mathfrak{g},\mathbb{R})$ as the $k$-th \emph{space of scalar cohomology} of the associative algebra $(\mathfrak{g},\circ)$. See \cite{Harrison,Hochschild} for details.

\begin{remark}
Within this framework, the flat scalar product $\langle \cdot,\cdot \rangle$ of a weakly flat pseudo-Riemannian $F$-Lie algebra $(\mathfrak{g},\circ,\langle \cdot,\cdot \rangle)$ can be interpreted as a scalar 2-cocycle in $H^2_{AS}(\mathfrak{g},\mathbb{R})$.
\end{remark}

Let $H^1_{AS}(\mathfrak{g}, \mathfrak{g})_r$ denote the first Hochschild cohomology space of the associative algebra $(\mathfrak{g},\circ)$ with coefficients in $\mathfrak{g}$. This space is taken with respect to the $\mathfrak{g}$-bimodule structure on $\mathfrak{g}$ where the left action is trivial and the right action is induced by the right multiplication, i.e., $r_x(y)=y\circ x$ for all $x,y\in \mathfrak{g}$.

\begin{lemma}\label{related1-cocycle}
Let $(\mathfrak{g},\circ)$ be an associative algebra equipped with a scalar product $\langle \cdot,\cdot \rangle$ satisfying the Frobenius identity with respect to the product $\circ$. Then, the formula
$$\theta(x,y)=\langle v(x),y \rangle,\quad x,y\in \mathfrak{g},$$
where $\theta \in H^2_{AS}(\mathfrak{g},\mathbb{R})$ and $v \in H^1_{AS}(\mathfrak{g}, \mathfrak{g})_r$, induces a linear isomorphism between $H^2_{AS}(\mathfrak{g},\mathbb{R})$ and $H^1_{AS}(\mathfrak{g}, \mathfrak{g})_r$.
\end{lemma}
\begin{proof}
Suppose $\theta\in Z^2_{AS}(\mathfrak{g},\mathbb{R})$. Then $\theta(x\circ y,z)= \theta(x,y\circ z)$ if and only if $\langle v(x\circ y),z \rangle=\langle v(x)\circ y,z \rangle$, for all $x,y,z\in \mathfrak{g}$. Since $\langle \cdot,\cdot \rangle$ is non-degenerate, the last equation implies
 $v(x\circ y)=v(x)\circ y$,  which means $v\in Z^1_{AS}(\mathfrak{g}, \mathfrak{g})_r$.\\

Now, assume $\theta \in B^2_{AS}(\mathfrak{g},\mathbb{R})$, so that it can be written as $\theta(x,y)=t(x\circ y)$ for some linear map $t:\mathfrak{g}\to \mathbb{R}$. Due to the non-degeneracy of  $\langle \cdot,\cdot \rangle$, there exists a unique element $x_0\in \mathfrak{g}$ such that $t=\langle x_0,\cdot \rangle$. This leads to $\langle v(x),y \rangle=\langle x_0\circ x,y \rangle$ for all $x,y\in \mathfrak{g}$. Therefore, $v(x)=x_0\circ x$, implying $v\in B^1_{AS}(\mathfrak{g}, \mathfrak{g})_r$.\\

It follows directly that the map $[\theta]\mapsto [v]$ establishes a well-defined linear isomorphism between $H^2_{AS}(\mathfrak{g},\mathbb{R})$ and $H^1_{AS}(\mathfrak{g}, \mathfrak{g})_r$. The inverse isomorphism can be constructed using analogous arguments.\\
\end{proof}

Building upon the work of \cite{AubertMedina}, we can establish a similar framework to define a reduction process for weakly flat pseudo-Riemannian $F$-Lie algebras.  We recall that a subspace $I$ of $(\mathfrak{g},\langle \cdot,\cdot \rangle)$ is termed \emph{totally isotropic} if $I\subset I^\perp$.

\begin{lemma}[Reduction]\label{Reduction}
Let $(\mathfrak{g},\circ,\langle \cdot,\cdot \rangle)$ be a weakly flat pseudo-Riemannian $F$-Lie algebra. Suppose $I$ is a totally isotropic subspace of dimension 1 in $(\mathfrak{g},\langle \cdot,\cdot \rangle)$ that is simultaneously a two-sided ideal of both $(\mathfrak{g},\circ)$ and $(\mathfrak{g},\ast)$. Then the following statements hold:
\begin{enumerate}
\item[i.] the products $\circ$ and $\ast$ within $I$ are null,  $I\circ I^\perp= I\ast I^\perp=0$, and $I^\perp$ is a two-sided ideal in $(\mathfrak{g},\circ)$, but only a left ideal in $(\mathfrak{g},\ast)$;
\item[ii.] $I^\perp$ is a right ideal in $(\mathfrak{g},\ast)$ if and only if $ I^\perp\ast I=0$; and
\item[iii.] if $I^\perp$ is a right ideal in $(\mathfrak{g},\ast)$, then clearly the canonical sequence
\begin{equation}\label{sequence}
0\to I\to I^\perp \to I^\perp/I\to 0,
\end{equation}
is a central exact sequence of left-symmetric and associative commutative subalgebras. The quotient vector space $B= I^\perp/I$ inherits a canonical structure of a weakly flat pseudo-Riemannian $F$-Lie algebra.
\end{enumerate}
\end{lemma}

\begin{proof}
Since properties i. and ii. regarding the Levi-Civita product $\ast$ have already been proven  in \cite{AubertMedina}, we focus solely on assertions involving the product $\circ$. Let  $x\in I$, $y\in I^\perp$ and $z\in\mathfrak{g}$.  The Frobenius identity and commutativity of $\circ$ ensure
$$\langle x\circ y ,z \rangle=\langle y\circ x ,z \rangle=\langle y,x\circ z \rangle=0,$$
since $x\circ z= z\circ x\in I$, implying $I\circ I^\perp=I^\perp \circ I=0$.  In particular, $I\circ I=0$.  Furthermore,
$$\langle z\circ y ,x \rangle=\langle y\circ z ,x \rangle=\langle y,z\circ x \rangle=0,$$
showing that $\mathfrak{g}\circ I^\perp=I^\perp\circ \mathfrak{g}\subset I^\perp$.  Hence,  $I^\perp$ is also a two-sided ideal of $(\mathfrak{g},\circ)$.\\

iii. Let us assume $I^\perp$  is a right ideal of $(\mathfrak{g},\ast)$. We denote elements of the quotient vector space $B=I^\perp/I$ by $\overline{x}:=x+I$. Based on the preceding results, $B$ inherits an associative and commutative product  $\underline{\circ}$ defined by $\overline{x} \underline{\circ} \overline{y}=x\circ y + I=\overline{x\circ y}$ and a left-symmetric product $\underline{\ast}$ defined $\overline{x} \underline{\ast} \overline{y}=x\ast y + I=\overline{x\ast y}$. Additionally, $B$ possesses a canonical Lie algebra structure given by
$$[\overline{x},\overline{y}]=\overline{x} \underline{\ast} \overline{y}-\overline{y} \underline{\ast} \overline{x}=(x\ast y-y\ast x)+I.$$

In particular, we get the canonical central exact sequence of left-symmetric and associative commutative subalgebras \eqref{sequence}. Note that the scalar product $\langle \cdot,\cdot \rangle$ induces a bilinear symmetric form on $I^\perp$ with kernel $I$ due to its totally isotropic nature. Consequently, restricting  $\langle \cdot,\cdot \rangle$  to $I^\perp\times I^\perp$ and passing to the quotient defines a flat scalar product  $\langle \cdot,\cdot \rangle'$ on $B$, explicitly given by $\langle \overline{x},\overline{y} \rangle'=\langle x,y\rangle$ for all $x,y\in I^\perp$. Reference \cite{AubertMedina} establishes that $\underline{\ast}$ coincides with the Levi-Civita product associated with $\langle \cdot,\cdot \rangle'$. Furthermore, we obtain
$$\langle \overline{x}\underline{\circ }\overline{y}, \overline{z}\rangle'= \langle\overline{x\circ y}, \overline{z} \rangle'=\langle x\circ y,z \rangle=\langle x,y\circ z \rangle =\langle\overline{x}, \overline{y\circ z} \rangle'= \langle \overline{x}, \overline{y}\underline{\circ }\overline{z}\rangle'.$$

Finally, similar arguments can be employed to verify the Hertling-Manin relation (Definition \ref{mainDef3}, iii.). This establishes a weakly flat pseudo-Riemannian F-Lie algebra $(B,\underline{\circ},\langle \cdot,\cdot \rangle')$.
\end{proof}

We shall refer to the quadruple $(B,\underline{\circ},\langle \cdot,\cdot \rangle')$ as the \emph{weakly flat pseudo-Riemannian F-reduction} of  $(\mathfrak{g},\circ,\langle \cdot,\cdot \rangle)$ with respect to the ideal $I$.\\

\begin{remark}
 In the case where $(\mathfrak{g},\circ,\langle \cdot,\cdot \rangle,e)$ is a flat pseudo-Riemannian $F$-Lie algebra, the quotient $B$ cannot inherit a canonical flat unit element with respect to the products $(\underline{\ast},\underline{\circ})$. This is because the element $e$ cannot belong to $I^\perp$.  If it were, we would have the contradiction $a\circ e=e \circ a=a\neq 0$, while $I\circ I^{\perp}=0$.  However, as shown in Proposition \ref{FlatUniryExtension}, we can still recover a flat unit element from the reduced weakly flat pseudo-Riemannian $F$-Lie algebra  $(B,\underline{\circ},\langle \cdot,\cdot \rangle')$ through a specific procedure.
\end{remark}

Let $(\mathfrak{g},\circ,\langle \cdot,\cdot \rangle)$ be a weakly flat pseudo-Riemannian $F$-Lie algebra satisfying the hypotheses of Lemma \ref{Reduction}. Then, according to the Nijenhuis cohomology (for left symmetric algebras) and the Hochschild cohomology (for associative algebras), the sequence \eqref{sequence} is described by the cohomology classes of a pair of scalar 2-cocycles $f\in Z^2_{SG}(\mathfrak{g},\mathbb{R})$ and $\theta\in Z^2_{AS}(\mathfrak{g},\mathbb{R})$ of the underlying left symmetric and associative commutative algebra $B$, respectively. Furthermore, if we set $I=\mathbb{R}a$ and identify $I^\perp$ with $\mathbb{R}a\oplus B$, the associative product $\circ$ and the left symmetric product $\ast$ on $I^\perp$, for all $x,y\in B$ and $\lambda,\nu \in\mathbb{R}$, are given by
$$(\lambda a+x)\circ (\nu a+y)=\theta(x,y)a+x\underline{\circ} y\qquad\textnormal{and}\qquad (\lambda a+x)\ast (\nu a+y)=f(x,y)a+x\underline{\ast} y.$$

Note that the commutativity of $\circ$ and $\underline{\circ}$ forces $\theta$ to be symmetric. By Lemma \ref{related1-cocycle}, the cohomology class of $\theta$ is determined by the cohomology class of a 1-cocycle $v\in Z^1_{AS}(B,B)_r$ (of the associative algebra $(B,\underline{\circ})$). This relationship is expressed by the formula $\theta(x,y)=\langle v(x),y \rangle'$. It follows that $\theta$ is symmetric if and only if $v$ is self-adjoint with respect to the scalar product $\langle \cdot,\cdot \rangle'$. In other words, $v^\ast=v$, where $v^\ast:B\to B$ denotes the adjoint map of $v:B\to B$  with respect to $\langle \cdot,\cdot \rangle'$.\\

On the other hand, following \cite{AubertMedina},  the cohomology class of $f$ is determined by the cohomology class of a 1-cocycle $u\in Z^1_{CE}(B,B)_L$ of the Lie algebra $B$ under the Lie algebra representation $L:B\to \mathfrak{gl}(B)$ defined by $L_x(y)=x\underline{\ast}y$ for all $x,y\in B$. This 1-cocycle is related to $f$ through the formula $f(x,y)=\langle u(x),y \rangle'$.\\

In conclusion:
\begin{lemma}\label{CentralCocycles}
The associative commutative product and the left symmetric product on $I^\perp=\mathbb{R}a\oplus B$ are respectively given by
$$(\lambda a+x)\circ (\nu a+y)=\langle v(x),y \rangle'a+x\underline{\circ} y\qquad\textnormal{and}\qquad (\lambda a+x)\ast (\nu a+y)=\langle u(x),y \rangle'a+x\underline{\ast} y,$$
for all $x,y\in B$ and $\lambda,\nu \in\mathbb{R}$, where $v:B\to B$ is a self-adjoint 1-cocycle in $Z^1_{AS}(B,B)_r$ of the associative commutative algebra $(B,\underline{\circ},\langle \cdot,\cdot \rangle')$ and $u:B\to B$ is a 1-cocycle in $Z^1_{CE}(B,B)_L$ of the Lie algebra $B$.
\end{lemma}

Let $s:B\to B$ be a linear map. We will denote its adjoint with respect to the scalar product $\langle \cdot,\cdot \rangle'$ by $s^\ast: B\to B$. Consider the line $\mathbb{R}d$ in $\mathfrak{g}$ that is dual to $I=\mathbb{R}a$, such that the hyperplane  $\mathbb{R}(a,d)$ is orthogonal to $B$. In other words, we have $\langle a,d\rangle=1$ and $\mathbb{R}(a,d)\perp B$. We can then identify $\mathfrak{g}$ with the direct sum $\mathbb{R}a\oplus B\oplus \mathbb{R}d$. Following \cite{AubertMedina}, we can rewrite the Levi--Civita product $\ast$ on $\mathfrak{g}$ as follows. Since $I$ and $I^\perp$ are two-sided ideals of $(\mathfrak{g},\ast)$, they are also Lie algebra ideals of $\mathfrak{g}$. Consequently, the Lie bracket on $\mathfrak{g}$ takes the following form
\begin{eqnarray}\label{LieBracketExtension}
	&  & [d,a]=\mu a\nonumber\\
	&  & [d,x]=-\langle b_0,x\rangle'a+D(x)\\
	&  & [x,y]=\langle(u-u^*)(x),y\rangle'a+[x,y]_B\nonumber,
\end{eqnarray}
where $b_0,x,y\in B$, $\mu\in \mathbb{R}$, $D\in\mathfrak{gl}(B)$.\\

Furthermore, the Levi--Civita product on $\mathfrak{g}$ is rewritten as
\begin{eqnarray}\label{Levi--CivitaProductExtension}
	&  & a\ast a=a\ast x=x\ast a=0\nonumber\\
	&  & x\ast y=\langle u(x),y\rangle' a+x\underline{\ast}y\nonumber\\
	&  & a\ast d=0\nonumber\\
	&  & d\ast a=\mu a\\
	&  & d\ast d=b_0-\mu d\nonumber\\
	&  & d\ast x=-\langle b_0,x\rangle'a+(D-u)(x)\nonumber\\
	&  & x\ast d=-u(x)\nonumber,
\end{eqnarray}
where $x,y \in B$, and $D\in \textnormal{Der}(B)$ is a Lie algebra derivation satisfying $(D-u)^\ast=-(D-u)$. Additionally, the following compatibility conditions must hold
\begin{equation}\label{Eq1}
[D,u]=u^2-\mu u-R_{b_0},
\end{equation}
where $R_{b_0}:B\to B$ is the linear map defined by $R_{b_0}(x)=x\underline{\ast} b_0$ for all $x\in B$, and, for all $x,y\in B$,
\begin{equation}\label{Eq2}
x\underline{\ast}u(y)-u(x\underline{\ast}y)=D(x)\underline{\ast}y+x\underline{\ast}D(y)-D(x\underline{\ast}y).
\end{equation}

For a detailed exposition, refer to \cite{AubertMedina}.\\

We now aim to ``decompose" the remaining structure on the weakly flat pseudo-Riemannian $F$-Lie algebra $(\mathfrak{g},\circ,\langle \cdot,\cdot\rangle)$. By Lemmas \ref{Reduction} and \ref{CentralCocycles}, the commutative product $\circ$ can be initially rewritten as follows
\begin{eqnarray*}
	&  & a\circ a=a\circ x=x\circ a=0\\
	&  & x\circ y=\langle v(x),y\rangle' a+x\underline{\circ}y\\
	&  & a\circ d=\lambda a\\
	&  & d\circ d=\beta a+a_0+\gamma d\\
	&  & d\circ x=\omega a+t(x),
\end{eqnarray*}
where $\lambda,\beta,\gamma,\omega\in \mathbb{R}$, $a_0, x, y \in B$ and $t\in \mathfrak{gl}(B)$. A straightforward computation reveals that  $\circ$ satisfies the Frobenius identity if and only if
$$t=v, \quad \lambda=\gamma,\quad\textnormal{and}\qquad \omega=\langle a_0,x\rangle'.$$

Furthermore, associativity of $\circ$ holds if and only if
\begin{equation}\label{Eq3}
v^2=\lambda v+r_{a_0},
\end{equation}
where $r_{a_0}:B\to B$ is the linear map defined by $r_{a_0}(x)=x\underline{\circ} a_0$ for all $x\in B$.\\

Therefore, we can express the product $\circ$ in its final form
\begin{eqnarray}\label{AssociativeCommutativeProductExtension}
	&  & a\circ a=a\circ x=x\circ a=0 \nonumber\\
	&  & x\circ y=\langle v(x),y\rangle' a+x\underline{\circ}y \nonumber\\
	&  & a\circ d=\lambda a\\
	&  & d\circ d=\beta a+a_0+\lambda d \nonumber\\
	&  & d\circ x=\langle a_0,x\rangle' a+v(x).\nonumber
\end{eqnarray}

We now turn our attention to the Hertling-Manin relation. We define a multilinear map:
\begin{eqnarray*}
	& HM(x,y,z,w)=  & [x\circ y,z\circ w]-[x\circ y,z]\circ w-[x\circ y,w]\circ z -x\circ [y,z\circ w]+x\circ [y,z]\circ w\\
	&  & + x\circ [y,w]\circ z-y\circ [x,z\circ w]+y\circ [x,z]\circ w+ y\circ [x,w]\circ z
\end{eqnarray*}
for all $x,y,z,w\in \mathfrak{g}$. We aim to determine the conditions under which $HM(x,y,z,w)=0$ holds for all quadruples $(x,y,z,w)$ of elements in  $\mathbb{R}a\oplus B\oplus \mathbb{R}d$, given the expressions for the Lie bracket \eqref{LieBracketExtension} and the associative commutative product \eqref{AssociativeCommutativeProductExtension}. Through lengthy, yet ultimately straightforward computations, we arrive at the following necessary algebraic conditions:
\begin{lemma}
If at least one of the components of the fourtuple $(x,y,z,w)$ equals ``$a$'', then $HM(x,y,z,w)=0$ trivially.
\end{lemma}
\begin{proof}

It follows directly from \eqref{LieBracketExtension} that the derived ideal $[\mathfrak{g},\mathfrak{g}]$ coincides with $\mathbb{R}a\oplus B$. Consequently, our assertion follows from the observations that  $[a,\mathbb{R}a\oplus B]=0$ and $a\circ (\mathbb{R}a\oplus B)=0$ along with $a\circ d=\lambda a$.
\end{proof}

We now present the necessary and sufficient conditions for the remaining cases, obtained by evaluating the Hertling-Manin tensor for specific quadruples.  For each case, we derive equations that must be satisfied by the parameters (structure coefficients) appearing also in \eqref{LieBracketExtension} and \eqref{AssociativeCommutativeProductExtension}.

\begin{itemize}
\item $HM(d,d,d,x)=0$ for all $x\in B$ if and only if

\begin{eqnarray}\label{eq1}
& &-\lambda\mu a_0+\lambda v(b_0)+v((u-u^\ast)(a_0))+v(D(a_0))+\lambda(u-u^\ast)(a_0)-\textnormal{ad}^\ast_{a_0}(a_0)-\lambda^2b_0\nonumber\\
& & \qquad+\lambda D^\ast(a_0)-2v(D^\ast(a_0))+2D^\ast(v(a_0))=0,
\end{eqnarray}
and 
\begin{equation}\label{eq2}
[\textnormal{ad}_{a_0},v]+\lambda [D,v]+2[v,vD]+r_{D(a_0)}=0.
\end{equation}\\

These equations also hold for any permutation of the quadruple $(d,d,d,x)$.\\

\item $HM(x,y,d,d)=0$ for all $x,y\in B$ if and only if

\begin{eqnarray}\label{eq3}
& & \lambda \mu v-r_{(u-u^\ast)(a_0)}-\lambda r_{b_0}+2r_{D^\ast(a_0)}-\lambda D^\ast v-2D^\ast r_{a_0}+\textnormal{ad}^\ast_{a_0}v\nonumber\\
& &\qquad+v\textnormal{ad}_{a_0}-\lambda vD-2r_{a_0}D=0,
\end{eqnarray}
and
\begin{equation}\label{eq4}
 [x\underline{\circ}y,a_0]_B-x\underline{\circ}[y,a_0]_B-y\underline{\circ}[x,a_0]_B+(2v-\lambda\textnormal{id}_B)(D(x\underline{\circ} y)-x\underline{\circ}D(y)-y\underline{\circ}D(x))=0. 
\end{equation}\\

Additionally, $HM(x,d,y,d)=0$ if and only if

\begin{eqnarray}\label{eq5}
& & v(u-u^\ast+D-D^\ast)v+\lambda [v,u-u^\ast]+\lambda(D^\ast v-vD)+D^\ast r_{a_0}-r_{a_0}D\nonumber\\
& & \qquad +\lambda^2(u-u^\ast)+\lambda\tilde{\textnormal{ad}}_{a_0}-(\tilde{\textnormal{ad}}_{a_0}v+v\tilde{\textnormal{ad}}_{a_0})+\tilde{\textnormal{ad}}_{v(a_0)}=0,
\end{eqnarray}
and
\begin{eqnarray}\label{eq6}
& & v(v([x,y]_B)-[v(x),y]_B-[v(y),x]_B)+v(x\underline{\circ}D(y)-y\underline{\circ}D(x))\nonumber\\
& &\qquad +D(v(x))\underline{\circ} y-D(v(y))\underline{\circ} x+[v(x),v(y)]_B=0.
\end{eqnarray}\\

Here, $\tilde{\textnormal{ad}}_{s}:B\to B$ denotes the linear map defined by $\tilde{\textnormal{ad}}_{s}(x)=\textnormal{ad}^\ast_x(s)$ for all $s, x\in B$.\\

The same equations are obtained for any permutation of the quadruple $(x,y,d,d)$.\\

\item $HM(x,y,z,d)=0$ for all $x,y,z\in B$ if and only if

\begin{eqnarray}\label{eq7}
& & (v-\lambda\textnormal{id}_B)((u-u^\ast)(x\underline{\circ} y))+v(D(x\underline{\circ} y)-x\underline{\circ}D(y)-y\underline{\circ}D(x))\nonumber\\
& & \qquad+(\lambda\textnormal{id}_B-v)(\textnormal{ad}_x^\ast(v(y))+\textnormal{ad}_y^\ast(v(x)))-\textnormal{ad}_{x\underline{\circ}y}^\ast(a_0)+\textnormal{ad}_{x}^\ast(a_0\underline{\circ}y)+\textnormal{ad}_{y}^\ast(a_0\underline{\circ}x)=0,
\end{eqnarray}
and
\begin{eqnarray}\label{eq8}
& & [x\overline{\circ}y,v(z)]_B-x\overline{\circ}[y,v(z)]_B-y\overline{\circ}[x,v(z)]_B+D(x\overline{\circ}y)\overline{\circ}z-(x\overline{\circ}D(y))\overline{\circ}z\nonumber\\
& &\qquad -(y\overline{\circ}D(x))\overline{\circ}z-v([x\overline{\circ}y,z]_B-x\overline{\circ}[y,z]_B-y\overline{\circ}[x,z]_B)=0.
\end{eqnarray}\\

These equations hold for any permutation of the quadruple $(x,y,z,d)$.\\

\item Finally, $HM(x,y,z,w)=0$ for all $x,y,z,w\in B$ if and only if

\begin{eqnarray}\label{eq9}
& & (u-u^\ast)(x\underline{\circ}y)\underline{\circ}z-v([x\overline{\circ}y,z]_B-x\overline{\circ}[y,z]_B-y\overline{\circ}[x,z]_B)-\textnormal{ad}_{x\underline{\circ}y}^\ast(v(z))\nonumber\\
& &\qquad-\textnormal{ad}_{x}^\ast(v(y))\underline{\circ }z-\textnormal{ad}_{y}^\ast(v(x))\underline{\circ }z+\textnormal{ad}_{x}^\ast(y\underline{\circ }v(z))+\textnormal{ad}_{y}^\ast(x\underline{\circ }v(z))=0,
\end{eqnarray}\\

and the Hertling--Manin relation must hold for $(B,[\cdot,\cdot]_B,\underline{\circ })$.\\
\end{itemize}

In conclusion, we can express the key findings in the following proposition.\\

\begin{proposition}\label{PropReduction}
	Let  $(\mathfrak{g},\circ,\langle \cdot,\cdot\rangle)$ be a weakly flat pseudo-Riemannian $F$-Lie algebra. Suppose $I$ is a totally isotropic subspace of dimension 1 in $(\mathfrak{g},\langle \cdot,\cdot \rangle)$. We further assume that $I$ is a two-sided ideal of both $(\mathfrak{g},\circ)$ and $(\mathfrak{g},\ast)$ and $I^\perp$ is a right ideal of $(\mathfrak{g},\ast)$. If $B=I^{\perp}/I$ denotes the associated weakly flat pseudo-Riemannian $F$-reduction of $\mathfrak{g}$ with respect to the ideal $I$, then:

	\begin{enumerate}
	\item[i.] the Levi--Civita product $\ast$ of $\mathfrak{g}$ is given by \eqref{Levi--CivitaProductExtension}, and
	\item[ii.] the associative commutative product $\circ$ of $\mathfrak{g}$ is given by \eqref{AssociativeCommutativeProductExtension}.
	\end{enumerate}
   
Moreover, the parameters (structure coefficients), $\mu,\lambda,\beta\in \mathbb{R}$, $a_0,b_0\in B$, $u\in Z^1_{CE}(B,B)_L$, $D\in \textnormal{Der}(B)$, $v\in Z^1_{AS}(B,B)_r$ with $(D-u)^\ast=-(D-u)$ and $v^\ast=v$, must satisfy the algebraic equations \eqref{Eq1} to \eqref{Eq3} and \eqref{eq1} to \eqref{eq9}.\\
\end{proposition}

Adding a flat unit to the structure yields the following proposition.

\begin{proposition}\label{FlatUniryExtension}
Under the same assumptions as Proposition \ref{PropReduction}, consider a flat pseudo-Riemannian F-Lie algebra  $(\mathfrak{g},\circ,\langle \cdot,\cdot\rangle,e)$ with a flat unit $e \in \mathfrak{g}$. Then, the flat unit $e$ can be expressed as 

\begin{equation}\label{eFlatUnity}
e=-\left( \frac{\langle a_0,\overline{e}\rangle'}{\lambda}+\frac{\beta}{\lambda^2}\right)a+\overline{e}+\frac{1}{\lambda}d,
\end{equation}
where $\overline{e}\in B$ satisfies

\begin{equation}\label{EqIdentityFlatness}
	v(\overline{e})=-\frac{1}{\lambda}a_0, \quad r_{\overline{e}}+\frac{1}{\lambda}v=\textnormal{id}_B,\quad (D-u)(\overline{e})=-\frac{1}{\lambda}b_0,\quad\textnormal{and}\quad R_{\overline{e}}-\frac{1}{\lambda}u=0.
\end{equation}

\end{proposition}

\begin{proof}
Suppose that $e=\gamma'a+\overline{e}+\omega'd$ for some $\gamma',\omega'\in \mathbb{R}$. The element $e \in \mathfrak g$ is a unit with respect to the product $\circ$ from \eqref{AssociativeCommutativeProductExtension} if and only if $\omega'=\frac{1}{\lambda}$, $\gamma'=-\left( \frac{\langle a_0,\overline{e}\rangle'}{\lambda}+\frac{\beta}{\lambda^2}\right)$ and the first two equalities in \eqref{EqIdentityFlatness} are satisfied. Additionally, $e$ is flat with respect to the Levi-Civita product \eqref{Levi--CivitaProductExtension} (that is, $\ast e=0$) if and only if $\mu=0$ and the remaining two equalities in \eqref{EqIdentityFlatness} hold true.\\

It is noteworthy that the derivation of these identities relied on the properties of the flat pseudo-Riemannian structures we are working with.\\
\end{proof}
 
Additionally, this result yields a method for constructing flat pseudo-Riemannian $F$-Lie algebras.\\

\begin{theorem}\label{ThmDoubleExtension}
Let $(B,\underline{\circ},\langle\cdot,\cdot \rangle')$ be a weakly flat pseudo-Riemannian $F$-Lie algebra. Suppose there exist parameters $\mu,\lambda,\beta\in \mathbb{R}$, $a_0,b_0, \overline{e}\in B$, $u\in Z^1_{CE}(B,B)_L$, $D\in \textnormal{Der}(B)$, $v\in Z^1_{AS}(B,B)_r$ with 
\begin{itemize}
\item $(D-u)^\ast=-(D-u)$,  $v^\ast=v$, and

\item satisfying the algebraic equations \eqref{Eq1} to \eqref{Eq3} and \eqref{eq1} to \eqref{eq9}.

\end{itemize}
Then, the vector space $\mathfrak{g}=\mathbb{R}a\oplus B\oplus \mathbb{R}d$ equipped with:
\begin{itemize}
\item a scalar product $\langle \cdot,\cdot\rangle$ extending $\langle\cdot,\cdot \rangle'$ ensuring $\mathbb{R}(a,b)$ is a hyperbolic plane orthogonal to $B$,

\item a Levi--Civita product $\ast$ given by \eqref{Levi--CivitaProductExtension}, and

\item an associative commutative product $\circ$ given by \eqref{AssociativeCommutativeProductExtension}
\end{itemize}
 defines a weakly flat pseudo-Riemannian $F$-Lie algebra.\\

Furthermore, if the additional algebraic equation  \eqref{EqIdentityFlatness} holds, implying $\mu=0$, then $\mathfrak{g}$ becomes a flat pseudo-Riemannian $F$-Lie algebra with a flat unit element $e$ given by \eqref{eFlatUnity}.\\
\end{theorem}

Motivated by Theorem \ref{ThmDoubleExtension}, we introduce the following definition.
\begin{definition}\label{doubleextensionFAS}
The weakly flat pseudo-Riemannian $F$-Lie algebra $(\mathfrak{g},\circ, \langle \cdot,\cdot\rangle)$ constructed in Theorem \ref{ThmDoubleExtension} is called the \emph{double extension} of the weakly flat pseudo-Riemannian $F$-Lie algebra $(B,\underline{\circ},\langle\cdot,\cdot \rangle')$  with respect to the parameters $(\mu, \lambda,\beta,a_0,b_0,u,D,v)$. \\

For the flat counterpart of the double extension, the parameter $\mu$ is set to zero, and a new parameter, $\overline{e} \in B$, is introduced.\\
\end{definition}

\begin{remark}\label{RemarkIntegration}
It is known that, as consequence of the Lie's third theorem, there exists a connected and simply Lie group $G$ integrating the Lie algebra $\mathfrak{g}$ constructed in Theorem \ref{ThmDoubleExtension}. Firstly, the parameters appearing in equation \eqref{LieBracketExtension} determine the structure constants of $\mathfrak{g}$ which, in turn, completely give rise to a parallelism formed by left invariant vector fields on $G$. Secondly, the scalar product $\langle \cdot,\cdot \rangle$ can be promoted to define a left invariant pseudo-Riemannian metric $\eta$ on $G$, so that the Levi--Civita product \eqref{Levi--CivitaProductExtension} is recovered when evaluating the corresponding Levi--Civita connection along the parallelism we have just fixed. In fact, the parameters that show up in equation \eqref{Levi--CivitaProductExtension} completely determine the Christoffel's symbols of the connection. Thirdly, the associative commutative product \eqref{AssociativeCommutativeProductExtension} yields a left invariant symmetric $(1,2)$-tensor field on $G$ by carrying out similar computations. Once again, the smooth component functions of such a tensor field are completely determined by the parameters appearing in equation \eqref{AssociativeCommutativeProductExtension}, after evaluating the parallelism of left invariant vector fields on $G$. All these geometric structures together define left invariant flat pseudo-Riemannian $F$-structures on $G$ if and only if the parameters of the double extension satisfy the conditions quoted in Theorem 3.9. As expected, such algebraic constrains can be rewritten in more geometric terms by using the associated smooth component functions that specify the geometric structures involved.
\end{remark}

A particularly noteworthy consequence of this double extension process is its ability to generate all weakly flat Lorentzian non-abelian bi-nilpotent $F$-Lie algebras having 1-dimensional light-cone subspaces that are associative two-sided ideals.\\

\begin{definition}
A \emph{bi-nilpotent Lie algebra} is a pair $(\mathfrak{g},\circ)$ where $\mathfrak{g}$ is a nilpotent Lie algebra and $\circ:\mathfrak{g}\times \mathfrak{g}\to \mathfrak{g}$ is a commutative associative product such that, for all $x\in \mathfrak{g}$, the right multiplication map $r_x:\mathfrak{g}\to \mathfrak{g}$ defined by $r_x(y)=x\circ y$ is a nilpotent linear map.
\end{definition}

We denote by $N(\mathfrak{g})=\lbrace x\in \mathfrak{g}: L_x=0\rbrace$ the kernel of the Lie algebra homomorphism $L:\mathfrak{g}\to\mathfrak{gl}(\mathfrak{g})$ defined by $L_x(y)=x\ast y$ for all $x,y\in\mathfrak{g}$. Additionaly, $Z(\mathfrak{g})=\lbrace x\in \mathfrak{g}: \textnormal{ad}_x=0\rbrace$ denotes the center of the Lie algebra $\mathfrak{g}$. The following fact follows from the proof of Theorem 3.1 in \cite{AubertMedina}.
\begin{lemma}
If $(\mathfrak{g},\langle\cdot,\cdot\rangle)$ is a flat Lorentzian non-abelian nilpotent Lie algebra. Then, there exists a non-zero element $\hat{a}\in  N(\mathfrak{g})\cap Z(\mathfrak{g})$ such that $\langle \hat{a},\hat{a}\rangle=0$.
\end{lemma}

Observe that the line $\mathbb{R}\hat{a}$ can be thought of as an 1-dimensional \emph{light-cone} subspace of $(\mathfrak{g},\langle\cdot,\cdot\rangle)$. Hence, in these terms we have:

\begin{theorem}\label{bi-Nil}
Let $(\mathfrak{g},\circ)$ be a non-abelian bi-nilpotent Lie algebra. Then, $\mathfrak{g}$ admits a structure of weakly flat Lorentzian $F$-Lie algebra $(\langle\cdot,\cdot\rangle,\circ)$, with $\mathbb{R}\hat{a}$ being a two-sided ideal of $(\mathfrak{g},\circ)$, if and only if it is obtained as a double extension of a weakly abelian bi-nilpotent Riemannian $F$-Lie algebra by a straight line with respect to $(0, 0,\beta,a_0,b_0,D,D,v)$ where $D^2=0$ and $\textnormal{im}(D)\subset (\mathbb{R}b_0)^\perp$. Moreover, $v$ is always nilpotent.
\end{theorem}

\begin{proof}
Suppose that $(\mathfrak{g},\circ, \langle \cdot,\cdot\rangle)$ is a weakly flat Lorentzian non-abelian bi-nilpotent $F$-Lie algebra such that $I=\mathbb{R}\hat{a}$ is a two-sided ideal of $(\mathfrak{g},\circ)$. It follows from \cite{AubertMedina} that $(\mathfrak{g},\langle \cdot,\cdot\rangle)$ is the pseudo-Riemannian double extension of an abelian Riemannian Lie algebra $(B,\langle \cdot,\cdot\rangle')$ by the straight line $I=\mathbb{R}\hat{a}$ with respect to $\mu=0$ and $u=D$ with $D^2=0$ and $\textnormal{im}(D)\subset (\mathbb{R}b_0)^\perp$. That is, $\mathfrak{g}=\mathbb{R}\hat{a}\oplus B\oplus \mathbb{R} d$. Since $B$ is abelian, it is clearly nilpotent and its Levi--Civita product $\underline{\ast}$ is trivial. By Lemma \ref{Reduction}, we can canonically induce an associative commutative product $\underline{\circ}$ over $B$ so that $(B,\underline{\circ}, \langle \cdot,\cdot\rangle')$ is a weakly abelian Riemannian $F$-Lie algebra. Thus, it remains to check that $\lambda=0$ and $\underline{\circ}$ is nilpotent. By assumption we know that $\circ$, which is given by \eqref{AssociativeCommutativeProductExtension}, is nilpotent. Additionally, $r_{\hat{a}}^k(d)=\lambda^k \hat{a}$ for all $k\in \mathbb{N}$, meaning that $\lambda=0$. Furthermore, we obtain the formulas:
$$r_x^k(y)=\langle v(x),\overline{r}_x^{k-1}(y)\rangle'\hat{a}+\overline{r}_x^{k}(y)\quad\textnormal{and}\quad r_d^k(x)=\langle a_0,v^{k-1}(x)\rangle'\hat{a}+v^{k}(x),\quad k\in \mathbb{N}$$
where $\overline{r}_x(y)=x\underline{\circ}y$ for all $x,y\in B$ and
$$r_d(d)=\beta a+a_0\quad \textnormal{and}\quad  r_d^k(d)=\langle a_0,v^{k-2}(a_0)\rangle'\hat{a}+v^{k-1}(a_0),\quad k\in \mathbb{N}_{\geq 2}.$$\\

Since $\circ$ is nilpotent, the product $\underline{\circ}$ must also be nilpotent.\\

Conversely, we now consider the double extension $\mathfrak{g}=\mathbb{R}\hat{a}\oplus B\oplus \mathbb{R} d$ of a weakly abelian bi-nilpotent Riemannian $F$-Lie algebra $(B,\underline{\circ},\langle\cdot,\cdot \rangle')$ with respect to $(0,0,\beta,a_0,b_0,D,D,v)$, where $D^2=0$ and $\textnormal{im}(D)\subset (\mathbb{R}b_0)^\perp$. From \cite{AubertMedina}, we know that such a $\mathfrak{g}$  is nilpotent, non-abelian and
admits a flat Lorentzian scalar product $\langle\cdot,\cdot \rangle$ with $\langle\hat{a},\hat{a} \rangle=0$. Therefore, since $\lambda=0$, Equation \eqref{Eq3} implies that $v^2=\overline{r}_{a_0}$. Consequently, the nilpotency of $\underline{\circ}$ implies the nilpotency of $v$. Hence, the formulas above show that $\circ$ must also be nilpotent. 
\end{proof}

\begin{corollary}\label{bi-NilCorollary}
There are no flat Lorentzian non-abelian bi-nilpotent $F$-Lie algebras $(\mathfrak{g},\circ, \langle \cdot,\cdot\rangle)$ such that $I=\mathbb{R}\hat{a}$ is a two-sided ideal of $(\mathfrak{g},\circ)$ for some non-zero element $\hat{a}\in  N(\mathfrak{g})\cap Z(\mathfrak{g})$ satisfying $\langle \hat{a},\hat{a}\rangle=0$.
\end{corollary}

\begin{proof}
This corollary follows directly from Proposition \ref{FlatUniryExtension}, becuase $\lambda$ must be zero.
\end{proof}

We can also prove that:

\begin{proposition}\label{bi-Nil2}
Let $(\mathfrak{g},\circ)$ be a non-abelian bi-nilpotent  Lie algebra. Then, $\mathfrak{g}$ admits a structure of weakly flat pseudo-Riemannian $F$-Lie algebra $(\langle\cdot,\cdot\rangle,\circ)$ with the signature of $\langle\cdot,\cdot\rangle$ equal to $(2,n-2)$ for $n\geq 4$, and $\mathbb{R}\hat{a}$ being a two-sided ideal of $(\mathfrak{g},\circ)$ for some non-zero $\hat{a} \in Z(\mathfrak{g})\cap Z(\mathfrak{g})^\perp$, if and only if it is obtained as a double extension of a weakly bi-nilpotent Lorentzian $F$-Lie algebra by a straight line with respect to $(0, 0,\beta,a_0,b_0,u,D,v)$ with $D$ nilpotent.
\end{proposition}

\begin{proof}
The proof follows a similar structure to that of Theorem \ref{bi-Nil}. However, we leverage the results established in \cite{BoucettaLebzioui2} concerning nilpotent Lie algebras endowed with flat scalar products of signature $(2,n-2)$ for $n\geq 4$.\\
\end{proof}

In analogy to Corollary  \ref{bi-NilCorollary}, the nonexistence of flat units also applies in this last case.

\section{Variations of the double extension}\label{S:TSC}

This section details the adaptation of our construction to yield two additional, closely related double extension processes applicable to the case of weakly flat pseudo-Riemannian $F$-Lie algebras. While the inclusion of flat units is admissible, their omission simplifies the discussion.

\subsection*{A stronger symmetric property} 

We revisit the symmetric property alluded to in Section \ref{S:FPRFS} (about weak Frobenius manifolds), which implies the Hertling–Manin relation. This condition is equivalent to the existence of a local potential over the associated simply connected and connected Lie group integrating the Lie algebra $\mathfrak{g}$. Despite the restrictive nature of this requirement, we can use similar techniques to construct a double extension process applicable to these specific algebraic objects.

Let $(\mathfrak{g},\circ, \langle \cdot,\cdot\rangle)$ be a real finite dimensional Lie algebra where $\circ:\mathfrak{g}\times \mathfrak{g}\to \mathfrak{g}$ is an associative commutative product and $\langle \cdot,\cdot\rangle$ is a flat scalar product on $\mathfrak{g}$ for which $\circ$ satisfies the Frobenius identity. Let us define the trilinear map $A(x,y,z)=\langle x\circ y,z\rangle$ for all $x,y,z\in\mathfrak{g}$ and consider the multi-linear map $\tilde{A}=\ast A$. Here $\ast$ stands for the Levi--Civita product associated to $\langle \cdot,\cdot\rangle$. The multi-linear map $\tilde{A}$ is explicitly given as
$$\tilde{A}(w,x,y,z)=-(\langle (w\ast x)\circ y,z \rangle+\langle x\circ(w\ast y),z \rangle+\langle x\circ y,w\ast z \rangle),$$
for all $w,x,y,z\in \mathfrak{g}$. We are particularly interested in cases where $\tilde{A}$ exhibits symmetry in all four arguments.\\

\begin{definition}
The triple $(\mathfrak{g},\circ, \langle \cdot,\cdot\rangle)$ above is said to be a Lie algebra of \emph{$F$-strong symmetric type} if $\tilde{A}(w,x,y,z)$ is symmetric in all four arguments $w,x,y,z\in \mathfrak{g}$.
\end{definition}

It can be readily verified that $\tilde{A}(w,x,y,z)$ is always symmetric under permutations of $x,y,z\in \mathfrak{g}$. Therefore, the $F$-strong symmetric property simplifies to verifying $\tilde{A}(w,\cdot,\cdot,\cdot)=\tilde{A}(\cdot,w,\cdot,\cdot)$ for all $w \in \mathfrak{g}$.\\

\begin{remark}
Lie algebras of $F$-strong symmetric type verify the Hertling--Manin relation from Definition \ref{mainDef3}. This is consequence of Theorem 2.15 in \cite{Hertling}. Thus, these types of algebraic objects give rise to examples of (weakly) flat pseudo-Riemannian $F$-Lie algebras. It is also worth mentioning that, after integrating the Lie algebra $\mathfrak{g}$ to a connected and simply connected Lie group $G$ (compare Remark \ref{RemarkIntegration}), we get an all-symmetry condition associated to the left invariant $(0,4)$-tensor field $(\nabla_WA)(X,Y,Z)$ where $A(X,Y,Z)=\eta(X\circ Y,Z)$. Additionally, as already mentioned in Section \ref{S:FPRFS}, such an all-symmetry condition is further equivalent to the existence of a left invariant potential $\Phi$ on $G$ verifying $(XYX)\Phi=A(X,Y,Z)$ for any flat vector fields $X,Y,Z$ (\cite[p. 22, Remarks 2.17]{Hertling}).

\end{remark}

Let $I$ be a totally isotropic subspace of dimension 1 in $(\mathfrak{g},\langle \cdot,\cdot \rangle)$ which is at the same time a two-sided ideal of $(\mathfrak{g},\circ)$ and $(\mathfrak{g},\ast)$ and such that $I^\perp$ is a right ideal of $(\mathfrak{g},\ast)$. It is simple to see that $\tilde{A}$ as well as the property of being symmetric in all four arguments pass to the quotient $B$ in Lemma \ref{Reduction}. Therefore, this allows us to leverage the double extension results from the previous section for $F$-strong symmetric Lie algebras. Namely, we only need to determine under what conditions $\tilde{A}(w,x,y,z)$ is symmetric in all four arguments $w,x,y,z\in \mathfrak{g}$. After using the expressions for both the Levi-Civita product \eqref{Levi--CivitaProductExtension} and the associative commutative product \eqref{AssociativeCommutativeProductExtension} we get that the equations below must be satisfied.

	\begin{itemize}
		\item If $w=a$ then simple computations show that $\tilde{A}(a,\cdot,\cdot,\cdot)=\tilde{A}(\cdot,a,\cdot,\cdot)=0$ since $a\ast (B\oplus \mathbb{R}d)=0$, $B\ast a=0$, $a\circ (B\oplus \mathbb{R}a)=0$, and $a\perp (B\oplus \mathbb{R}a)$.\\

		\item Let us consider the case $w=d$. Firstly, it follows that $\tilde{A}(d,a,d,d)= \tilde{A}(a,d,d,d)$ if and only if $-\lambda\mu =0$. Secondly, $\tilde{A}(d,x,y,z)= \tilde{A}(x,d,y,z)$ for all $x,y,z\in B$ if and only if
		
		\begin{equation}\label{EqS1}
			(D-u)(x)\underline{\circ} y+x\underline{\circ} (D-u)(y)-(D-u)(x\underline{\circ} y)=-u(x)\underline{\circ} y+v(x\underline{\ast}y)-x\underline{\ast}v(y),\\
		\end{equation}
		
		for all $x,y\in B$. Thirdly, $\tilde{A}(d,x,y,d)= \tilde{A}(x,d,y,d)$ for all $x,y,z\in B$ if and only if 
			
		\begin{equation}\label{EqS2}
			[v,D-u]=-2vu +\lambda u+\mu v-R_{a_0}-r_{b_0}.
		\end{equation}
		
		Fourthly, $\tilde{A}(d,x,d,d)= \tilde{A}(x,d,d,d)$ for all $x\in B$ if and only if
		\begin{equation}\label{EqS3}
			-\lambda b_0 -(D-u)(a_0)+2v(b_0)-2\mu a_0=-3u^\ast(a_0).
		\end{equation}
		
		The other possible cases hold true trivially. \\
		
		\item We now analyze the remaining cases for $w\in B$. Firstly, the identity $\tilde{A}(w,x,y,z)=\tilde{A}(x,w,y,z)$ holds true for all $x,y,z\in B$ because, due to Lemma \ref{Reduction}, the multi-linear map $\tilde{A}'=\underline{\ast}A'$ associated to the triple $(\underline{\ast},\underline{\circ},\langle\cdot,\cdot\rangle'
		)$ on $B$ already satisfies the symmetry property we are verifying. Secondly, $\tilde{A}(w,x,y,d)=\tilde{A}(x,w,y,d)$ if and only if
		\begin{equation}\label{EqS4}
			v(x\underline{\ast}y)-x\underline{\ast} v(y)-y\underline{\circ}u(x)=v(y\underline{\ast}x)-y\underline{\ast} v(x)-x\underline{\circ}u(y),
		\end{equation}
		for all $x,y\in B$. Thirdly, $\tilde{A}(w,x,d,d)=\tilde{A}(x,w,d,d)$ if and only if
		\begin{equation}\label{EqS5}
			\lambda u(x)-R_{a_0}(x)-2v(u(x))=\lambda u^\ast(x)+R^\ast_x(a_0)-2u^\ast(v(x)),
		\end{equation}
		for all $x\in B$. The other possible cases hold true trivially. 
\end{itemize}

Therefore, based on the derived conditions, we can now rewrite the Lie algebra structure of $F$-strong symmetric type over $\mathfrak{g}$ similarly as we did in Proposition \ref{PropReduction}. Furthermore, this analysis provides a method for constructing Lie algebras of $F$-strong symmetric type.\\

\begin{proposition}
	Let $(B,\underline{\circ},\langle\cdot,\cdot \rangle')$ be a Lie algebra of $F$-strong symmetric type. Assume that there are parameters $\mu,\lambda,\beta\in \mathbb{R}$, $a_0,b_0\in B$, $u\in Z^1_{CE}(B,B)_L$, $D\in \textnormal{Der}(B)$, $v\in Z^1_{AS}(B,B)_r$ satisfying that
\begin{itemize}
\item $\lambda \mu=0$, $(D-u)^\ast=-(D-u)$, $v^\ast=v$, and

\item the equations \eqref{Eq1} to \eqref{Eq3} and \eqref{EqS1} to \eqref{EqS5} hold.
\end{itemize}
Then, the vector space $\mathfrak{g}=\mathbb{R}a\oplus B\oplus \mathbb{R}d$ equipped with 
\begin{itemize}
\item a scalar product $\langle \cdot,\cdot\rangle$ which extends $\langle\cdot,\cdot \rangle'$ verifying that $\mathbb{R}(a,d)$ is a hyperbolic plane orthogonal to $B$,

\item a Levi--Civita product $\ast$ given by \eqref{Levi--CivitaProductExtension}, and

\item an associative commutative product $\circ$ given by  \eqref{AssociativeCommutativeProductExtension},

\end{itemize}

becomes a Lie algebra of $F$-strong symmetric type.\\

\end{proposition}

The Lie algebra of $F$-strong symmetric type $(\mathfrak{g},\circ, \langle \cdot,\cdot\rangle)$  constructed herein is called the \emph{double extension} of the Lie algebra of $F$-strong symmetric type $(B,\underline{\circ},\langle\cdot,\cdot \rangle')$ with respect to $(\mu, \lambda,\beta,a_0,b_0,u,D,v)$. \\

\begin{remark}
Similar results as those stated in Theorem \ref{bi-Nil} and Proposition \ref{bi-Nil2} can be also stated for bi-nilpotent Lie algebras of $F$-strong symmetric type.
\end{remark}

\subsection*{Frobenius-like Poisson algebras}

Let us now explain how to obtain a double extension for Poisson algebras compatible with a flat scalar product. It is worth noting that the forthcoming construction can also be applied when dealing with complex Poisson algebras.\\

\begin{definition}
A real Lie algebra $\mathfrak{g}$ is said to be a \emph{Poisson algebra} if it can be endowed with an associative and commutative product $\circ:\mathfrak{g}\times \mathfrak{g}\to \mathfrak{g}$ such that the following Leibniz condition is satisfied:
\begin{equation}\label{LaibnizPoisson}
[x,y\circ z]=[x,y]\circ z+[x,z]\circ y ,\qquad x,y,z\in \mathfrak{g}.\\
\end{equation}

The Lie bracket acts as a derivation of the associative commutative product. Therefore, the Leibnizator $\mathcal{L}(x,y,z)=[x,y\circ z]-[x,y]\circ z-[x,z]\circ y$ vanishes for all $x,y,z\in \mathfrak{g}$. Additionally, if $\mathfrak{g}$ admits a flat scalar product $\langle \cdot,\cdot \rangle$ for which $\circ$ satisfies the Frobenius identity then we refer to the triple $(\mathfrak{g},\circ,\langle \cdot,\cdot \rangle)$ as a \emph{Frobenius-like Poisson algebra}.
\end{definition}

\begin{remark}
Since the Leibnizator identically vanishes in Frobenius-like Poisson algebras, they trivially satisfy the Hertling-Manin relation and thus exemplify weakly flat pseudo-Riemannian $F$-Lie algebras.
\end{remark}

We consider a Frobenius-like Poisson algebra $(\mathfrak{g},\circ,\langle \cdot,\cdot\rangle)$ with a totally isotropic subspace $I$ of dimension 1 that satisfies the following properties: $I$ is a two-sided ideal of both $(\mathfrak{g},\circ)$ and $(\mathfrak{g},\ast)$, and $I^\perp$ is a right ideal of $(\mathfrak{g},\ast)$. Building upon Lemma \ref{Reduction}, it not hard to see that the Leibniz identity \eqref{LaibnizPoisson} holds on the quotient space $B=I^\perp/I$ so that $B$ inherits the structure of a  Frobenius-like Poisson algebra $(\underline{\circ},\langle \cdot,\cdot\rangle')$.  We maintain the identification $\mathfrak{g}$ with $\mathbb{R}a\oplus B\oplus \mathbb{R}d$ established previously (with \eqref{LieBracketExtension} for the Lie bracket and \eqref{AssociativeCommutativeProductExtension} for the associative commutative product).\\

The crucial result is that the Leibniz identity \eqref{LaibnizPoisson} holds on $\mathfrak g$, or equivalently the Leibnizator vanishes, if and only if the following conditions are satisfied.\\

\begin{itemize}
\item Let $x,y \in \mathfrak g$, then $\mathcal{L}(d,d,a)=0$ if and only if $\lambda \mu=0$. Additionaly, we obtain $\mathcal{L}(x,y,a)=0=\mathcal{L}(x,a,y)$, where we have used $[a,\mathbb{R}a\oplus B]=0$, $a\circ (\mathbb{R}a\oplus B)=0$, and $a\circ d=\lambda a$.\\

\item For all $x,y\in B$, $\mathcal{L}(x,y,d)=0$ if and only if
\begin{equation}\label{EqP1}
(u-u^\ast)(v-\lambda\textnormal{id}_B)=\tilde{\textnormal{ad}}_{a_0}+D^\ast v\quad\textnormal{and}\quad v([x,y]_B)=[v(x),y]_B-D(y)\underline{\circ} x.
\end{equation}

The same equations hold for the triple $(x,d,y)$.\\

\item For all $x,y\in B$, $\mathcal{L}(d,x,y)=0$  if and only if
\begin{equation}\label{EqP2}
	-\mu v+vD+D^\ast v=-r_{b_0}\quad\textnormal{and}\quad D(x\underline{\circ}y)=D(x)\underline{\circ}y+x\underline{\circ}D(y),
\end{equation}

 implying $D$ is derivation of $(B,\underline{\circ})$.\\

\item For all $x\in B$, $\mathcal{L}(d,x,d)=0$ if and only if
\begin{equation}\label{EqP3}
	-\mu a_0+v(b_0)=\lambda b_0-D^\ast(a_0)\quad\textnormal{and}\quad [D,v]=0.
\end{equation}

The same identities hold for the triple $(d,d,x)$.\\

\item For all $x\in B$, $\mathcal L(x, d,d)=0$  if and only if
\begin{equation}\label{EqP4}
	(u-u^\ast-2D^\ast)(a_0)=-\lambda b_0\quad\textnormal{and}\quad \textnormal{ad}_{a_0}=(2v-\lambda\textnormal{id}_B)D.\\
\end{equation}

\item For all $x,y,z\in B$, $\mathcal L(x,y,z)=0$ if and only if
\begin{equation}\label{EqP5}
(u-u^\ast)(x\underline{\circ} y)=\textnormal{ad}^\ast_y(v(x))+\textnormal{ad}^\ast_x(v(y)).
\end{equation}

Furthermore, the Leibniz identity \eqref{LaibnizPoisson} is satisfied in $(B,[\cdot,\cdot]_B,\underline{\circ })$.\\
\item Finally, $\mathcal L(d,d,d)=0$ if and only if
\begin{equation}\label{EqP6}
-\mu \beta+\langle a_0,b_0\rangle'=0\quad\textnormal{and}\quad D(a_0)=0.\\
\end{equation}
\end{itemize}

This result establishes a method for constructing new $n$-dimensional Frobenius-like Poisson algebras from a given $(n-2)$-dimensional one satisfying specific properties.\\

\begin{proposition}
Let $(B,\underline{\circ},\langle\cdot,\cdot \rangle')$ be a Frobenius-like Poisson algebra. Suppose there exist parameters $\mu,\lambda,\beta\in \mathbb{R}$, $a_0,b_0\in B$, $u\in Z^1_{CE}(B,B)_L$, $D\in \textnormal{Der}(B)$, $v\in Z^1_{AS}(B,B)_r$ satisfying the following conditions:
\begin{itemize}
\item $\lambda \mu=0$, $(D-u)^\ast=-(D-u)$, $v^\ast=v$, and

\item the algebraic equations \eqref{Eq1} to \eqref{Eq3} and \eqref{EqP1} to \eqref{EqP6} hold.
\end{itemize}

Then,  the vector space $\mathfrak{g}=\mathbb{R}a\oplus B\oplus \mathbb{R}d$ equipped with
\begin{itemize}
\item a scalar product $\langle \cdot,\cdot\rangle$ that extends $\langle\cdot,\cdot \rangle'$, such that the subspace $\mathbb{R}(a,d)$ is a hyperbolic plane orthogonal to $B$,

\item a Lie bracket $[\cdot,\cdot]$ defined by \eqref{LieBracketExtension}, and

\item an associative commutative product $\circ$ defined by \eqref{AssociativeCommutativeProductExtension}
\end{itemize}
becomes a Frobenius-like Poisson algebra.

\end{proposition}

We refer to the Frobenius-like Poisson algebra $(\mathfrak{g},\circ, \langle \cdot,\cdot\rangle)$ constructed in the preceding result as the double extension of the Frobenius-like Poisson algebra $(B,\underline{\circ},\langle\cdot,\cdot \rangle')$ with respect to the parameters $(\mu, \lambda,\beta,a_0,b_0,u,D,v)$.\\

\begin{remark}
Similar results as those established in Theorem \ref{bi-Nil} and Proposition \ref{bi-Nil2} can be proved to apply for Frobenius-like bi-nilpotent Poisson algebras.
\end{remark}

\section{Examples}\label{S:E}

We now illustrate the key constructions developed throughout the preceding sections with a series of pertinent examples.\\

\begin{example}
Up to isomorphism, there exists a unique non-abelian Lie algebra of dimension 2. This Lie algebra, denoted by $\mathfrak{aff}(\mathbb{R})=\mathbb{R}(a,d)$, is equipped with the Lie bracket defined by $[d,a]=a$. Furthermore, it admits a flat Lorentzian scalar product given by $\langle d,a \rangle=1$.\\

The Levi-Civita product $\ast$ and the associative commutative product $\circ$ constructed via the double extension process are given explicitly as:

	$$\begin{tabular}{l | c  r}
	$\ast$ & $d$ & $a$ \\
	\hline
	$d$ & $-d$ & $a$\\
	$a$ & $0$ & $0$\\
\end{tabular}\qquad\qquad\textnormal{and}\qquad\qquad \begin{tabular}{l | c  r}
	$\circ$ & $d$ & $a$ \\
	\hline
	$d$ & $\beta a+\lambda d$ & $\lambda a$\\
	$a$ & $\lambda a$ & $0$\\
\end{tabular}
$$

By setting either $\lambda=1$ or $\lambda=0$, we obtain the only two possibilities for weakly flat pseudo-Riemannian non-abelian $F$-Lie algebras in dimension 2 with a non-trivial commutative associative product.  This classification aligns with the existing classifications for associative commutative algebras \cite{KobayashiShirayanagiTsukadaTakahasi,RemmGoze}  and flat pseudo-Riemannian Lie algebras \cite{AubertMedina} in this dimension.  It is important to note that if $\lambda=1$ and $\beta=0$, the element $d$ assumes the role of a unit element. Nonetheless, it is not flat.\\

For the remaining two realizations of the double extension process (the $F$-strong symmetric case and the Frobenius-like Poisson algebras), we are compelled to restrict $\lambda$ to 0. This constraint simplifies the product $\circ$ to:
$$\begin{tabular}{l | c  r}
	$\circ$ & $d$ & $a$ \\
	\hline
	$d$ & $\beta a$ & $0$\\
	$a$ & $0$ & $0$\\
\end{tabular}$$

Furthermore, setting $\beta\neq 0$ recovers the unique non-trivial and non-abelian Poisson algebra in dimension 2, as detailed in \cite{GozeRemm}.

\end{example}

\begin{example}
The 3-dimensional Heisenberg Lie algebra $\mathfrak{h}_3=\mathbb{R}(a,x,d)$, with Lie bracket $[d,x]=-a$, admits a flat scalar product (see \cite{AubertMedina}). The flat scalar product $\langle\cdot,\cdot\rangle$ and its Levi-Civita product $\ast$ are given by:

$$\langle\cdot,\cdot\rangle= \left( 
\begin{array}{ccc} &  & 1 \\
	& 1 &\\
1	& & 
\end{array}%
\right)\qquad\leadsto \qquad \begin{tabular}{l | c  r  c  r}
	$\ast$ & $a$ & $x$ & $d$ \\
	\hline
	$a$ & $0$ & $0$ & $0$\\
	$x$ & $0$ & $0$ & $0$\\
	$d$ & $0$ & $-a$ & $x$\\
\end{tabular}$$

It can be constructed via the double extension process applied to the 1-dimensional real vector space $\mathbb{R}x$ equipped with scalar product $\langle x,x\rangle=1$ with respect to the parameters  $u=D=0$, $\mu=0$ and $b_0=x$.  Substituting these parameters into Equation \eqref{eq1} yields two possibilities: $\lambda=0$ or $v(x)=\lambda x$. Equation \eqref{Eq3} further clarifies that if $\lambda=0$, then the associative commutative product on $\mathbb{R}x$ becomes $x\underline{\circ}x=v^2(x)$. In this case, $v\in Z^1_{AS}(\mathbb{R}x,\mathbb{R}x)_r$ if and only if $v=0$ or $v=\textnormal{id}_{\mathbb{R}x}$. Otherwise, if $v(x)=\lambda x$, the product $\underline{\circ}$ becomes $x\underline{\circ}x=0$.\\

It is noteworthy that due to the abelian nature of $\mathbb{R}x$, all the equations from \eqref{eq2} to \eqref{eq9} are automatically satisfied. Consequently, by applying the double extension, we acquire three parametric families of weakly flat Lorentzian $F$-Lie algebra structures on $\mathfrak{h}_3$. These structures are characterized by the following associative commutative products\footnote{$\circ_1$ for $\lambda=0$ and $v=0$; $\circ_2$ for $\lambda=0$ and $v=\textnormal{id}_{\mathbb{R}x}$; and $\circ_3$ for $v(x)=\lambda x$ with $\lambda \neq 0$.}:

$$\begin{tabular}{l | c  r  c  r}
	$\circ_1$ & $a$ & $x$ & $d$ \\
	\hline
	$a$ & $0$ & $0$ & $0$\\
	$x$ & $0$ & $0$ & $\alpha a$\\
	$d$ & $0$ & $\alpha a$ & $\beta a+\alpha x$\\
\end{tabular}
\qquad\begin{tabular}{l | c  r  c  r}
$\circ_2$ & $a$ & $x$ & $d$ \\
\hline
$a$ & $0$ & $0$ & $0$\\
$x$ & $0$ & $x$ & $\alpha a+x$\\
$d$ & $0$ & $\alpha a+x$ & $\beta a+\alpha x$\\
\end{tabular}\qquad \begin{tabular}{l | c  r  c  r}
$\circ_3$ & $a$ & $x$ & $d$ \\
\hline
$a$ & $0$ & $0$ & $\lambda a$\\
$x$ & $0$ & $0$ & $\alpha a+\lambda x$\\
$d$ & $\lambda a$ & $\alpha a+\lambda x$ & $\beta a+\alpha x+\lambda d$\\
\end{tabular}$$
where  $\alpha, \beta, \lambda \in \mathbb{R}$, with the constraint $\lambda\neq 0$. Observe that for $\circ_3$, if $\lambda=1$ and $\alpha=\beta=0$, then element $d$ becomes a unit element. However, for this specific configuration $d$ is not flat.\\

We now proceed to examine the remaining two double extension processes associated with the Heisenberg Lie algebra $\mathfrak{h}_3$. In the $F$-strong symmetric case, from Equation \eqref{EqS2}, we can infer that $x\underline{\circ} x=0$. Additionally, Equation \eqref{EqS3} implies that $v(x)=\frac{\lambda}{2}x$. Since $\underline{\ast}=0$, Equations \eqref{EqS1} to \eqref{EqS5} are trivially satisfied. Therefore, this case yields a parametric family of Lie algebra structures of $F$-strong symmetric type over $\mathfrak{h}_3$. These structures are characterized by the associative commutative product:
$$ \begin{tabular}{l | c  r  c  r}
	$\circ_4$ & $a$ & $x$ & $d$ \\
	\hline
	$a$ & $0$ & $0$ & $\lambda a$\\
	$x$ & $0$ & $0$ & $\alpha a+\frac{\lambda}{2}x$\\
	$d$ & $\lambda a$ & $\alpha a+\frac{\lambda}{2}x$ & $\beta a+\alpha x+\lambda d$\\
\end{tabular}$$
where $\alpha,\beta,\lambda\in \mathbb{R}$. \\

For the Frobenius-like Poisson algebras case, upon substituting the initial parameters ($u=0=D, \mu=0, b_0=x$) into all the equations from \eqref{EqP1} to \eqref{EqP6}, we establish that $\mathfrak{h}_3$ admits a one parametric family of Frobenius-like Poisson structures. This family is characterized by the associative commutative product:
$$ \begin{tabular}{l | c  r  c  r}
	$\circ_5$ & $a$ & $x$ & $d$ \\
	\hline
	$a$ & $0$ & $0$ & $0$\\
	$x$ & $0$ & $0$ & $0$\\
	$d$ & $0$ & $0$ & $\beta a$\\
\end{tabular}$$
where $\beta\in \mathbb{R}$.

\end{example}

Let us now exhibit a simple manner to get interesting examples in higher dimensions.

\begin{example}
We consider the abelian Lie algebra $B=\mathbb{R}^n$  equipped with an arbitrary signature scalar product $\langle\cdot,\cdot\rangle'$, resulting in a trivial Levi-Civita product ($\underline{\ast}=0$). We analyze two specific cases of the double extension process.\\

\textbf{Case 1}: $D=0=u$. In this scenario, the Lie bracket on $\mathfrak{g}=\mathbb{R}a\oplus \mathbb{R}^n\oplus \mathbb{R}d$ simplifies to $[d,a]=\mu a$ and $[d,x]=-\langle b_0,x\rangle'a$ where $b_0\in \mathbb{R}^n$ and $\mu\in\mathbb{R}$. Furthermore, the Levi--Civita product  takes the following form:

$$ \begin{tabular}{l | c  r  c  r}
	$\ast$ & $a$ & $y$ & $d$ \\
	\hline
	$a$ & $0$ & $0$ & $0$\\
	$x$ & $0$ & $0$ & $0$\\
	$d$ & $\mu a$ & $-\langle b_0,y\rangle'a$ & $b_0-\mu d$\\
\end{tabular}$$

Our objective is reduced to identify associative commutative products $\underline{\circ}$ on $\mathbb{R}^n$ that satisfy the following system of equations:
\begin{eqnarray*}
v^2-\lambda v-r_{a_0} & = & 0,\\
-\lambda\mu a_0+\lambda v(b_0)-\lambda^2b_0& = &0\\
\lambda\mu v-\lambda r_{b_0}& = & 0,
\end{eqnarray*}
where $v\in Z^1_{AS}(\mathbb{R}^n,\mathbb{R}^n)_r$, $a_0,b_0\in\mathbb{R}^n$, and $\lambda,\mu\in \mathbb{R}$. \\

We can consider two particular subcases:
\begin{itemize}
\item $\lambda=0$: In this case, the first equation above, becomes $v^2=r_{a_0}$, and the other equations hold trivially.

\item $\lambda\neq 0$: If $a_0=0$, the first equation above, becomes $v(v-\lambda\textnormal{id}_{\mathbb{R}^n})=0$. Additionally, $v(b_0)=\lambda b_0$ and $\mu v=r_{b_0}$ from the other two equations. A specific solution is $\mu=0$ and  $v=\lambda\textnormal{id}_{\mathbb{R}^n}$, leading to $r_{b_0}=0$. This case results in flat units only when $\overline{e}=0$ and $b_0=0$, forcing $\mathfrak{g}$ to be abelian as well.\\
\end{itemize}

The double extension of Lie algebras of $F$-strong symmetric type is parametrized by the equations:

\begin{eqnarray*}
	\mu v-r_{b_0} & = & 0\\
	-\lambda b_0+2 v(b_0)-2\mu a_0& = &0
\end{eqnarray*}
where $v\in Z^1_{AS}(\mathbb{R}^n,\mathbb{R}^n)_r$, $a_0,b_0\in\mathbb{R}^n$, and $\lambda,\mu\in \mathbb{R}$ with $\lambda\mu=0$. \\

Moreover, the double extension of Frobenius-like Poisson algebras is parametrized by the equations:

\begin{eqnarray*}
	\mu v-r_{b_0} & = & 0\\
	-\mu a_0+v(b_0)& = &0\\
	-\mu\beta +\langle a_0,b_0\rangle'& = &0\\
	\lambda b_0& = & 0
\end{eqnarray*}
where $v\in Z^1_{AS}(\mathbb{R}^n,\mathbb{R}^n)_r$, $a_0,b_0\in\mathbb{R}^n$, and $\lambda,\mu,\beta\in \mathbb{R}$ with $\lambda\mu=0$.\\

An analogous analysis can be performed for specific parameter choices in each case.\\

\textbf{Case 2}: $u=u^\ast$ and $v=0$. Under this scenario, Equation \eqref{Eq3} implies $r_{a_0}=0$. Additionally, $D^\ast=2u-D$ and Equation \eqref{Eq1} must be satisfied. The Lie bracket on $\mathfrak{g}=\mathbb{R}a\oplus \mathbb{R}^n\oplus \mathbb{R}d$ becomes: $[d,a]=\mu a$ and $[d,x]=-\langle b_0,x\rangle'a+D(x)$, where $b_0\in \mathbb{R}^n$. The Levi--Civita product on $\mathfrak{g}$ takes the following form:
$$ \begin{tabular}{l | c  r  c  r}
	$\ast$ & $a$ & $y$ & $d$ \\
	\hline
	$a$ & $0$ & $0$ & $0$\\
	$x$ & $0$ & $\langle u(x),y\rangle'a$ & $-u(x)$\\
	$d$ & $\mu a$ & $-\langle b_0,y\rangle'a+(D-u)(y)$ & $b_0-\mu d$\\
\end{tabular}$$
\end{example}

We now focus on associative and commutative products $\underline{\circ}$ on $\mathbb{R}^n$ that satisfy $r_{a_0}=0$. For simplicity, let us further restrict ourselves to the case where $D$ acts as a derivation on the associative product $\underline{\circ}$. Therefore, these products must additionally fulfill the following system of equations:

\begin{eqnarray*}
	-\lambda\mu a_0-\lambda^2b_0+\lambda (2u-D)(a_0)& = &0 \\
	-\lambda r_{b_0}+2r_{(2u)(a_0)}& = & 0\\
	r_{D(a_0)}& = &0\\
\end{eqnarray*}
where $u\in \mathfrak{gl}(\mathbb{R}^n)$, $D\in\textnormal{Der}(\mathbb{R}^n,\underline{\circ})$, $a_0,b_0\in \mathbb{R}^n$, and $\mu,\lambda\in\mathbb{R}$.\\

We explore two specific subcases:

\begin{itemize}
\item $u=0$ and $a_0\in \textnormal{ker}(D)$: In this scenario, the system of equations reduces to $\lambda(\mu a_0+\lambda b_0)=0$ and $\lambda r_{b_0}=0$. Therefore, either $\lambda=0$ or $\lambda\neq 0$ with $b_0=-\frac{\mu}{\lambda}a_0$. The second condition directly imply $r_{b_0}=0$. Additionally, if $a_0=0$ and $\lambda\neq 0$, we obtain a flat unit provided $\overline{e}$ is a unit for $\underline{\circ}$ within the kernel of $D$. 

\item $u=\frac{1}{2}D$ and $a_0\in \textnormal{ker}(D)$:  Here, the equation $\frac{1}{4}D^2- \frac{1}{2}\mu D=0$ holds. Interestingly, the resulting solutions for the system are identical to those in the previous subcase.\\
\end{itemize}

For the double extension of Lie algebras of $F$-strong symmetric type we further assume that $D\in \textnormal{Der}(\mathbb{R}^n,\underline{\circ})$ and $u=0$. This implies that $r_{b_0}=0$, so that the construction is parametrized by the equation
$$-\lambda b_0-D(a_0)-2\mu a_0=0,$$
where $a_0,b_0\in\mathbb{R}^n$ and $\lambda,\mu\in \mathbb{R}$ with $\lambda\mu=0$. \\

Furthermore, if $D=0$, the double extension of Frobenius-like Poisson algebras must satisfy $\mu a_0=0$ and $\lambda b_0=0$ where $a_0,b_0\in\mathbb{R}^n$ and $\lambda,\mu,\beta\in \mathbb{R}$ with $\lambda\mu=0$.\\

Note that the preceding analysis for the abelian Lie algebra $\mathbb{R}^n$ allows us to construct examples of flat pseudo-Riemannian $F$-Lie algebras in dimensions $n+2$ for all $n\geq 0$. This can be achieved by leveraging the existing classifications of associative commutative algebras $(\mathbb{R}^n,\underline{\circ})$ documented in the literature (e.g., \cite{DeGraaf,GozeRemm} and its cited references).\\

%\vspace*{0.5cm}
%\noindent {\bf Conflict of interests.} The authors declare that they have no conflict of interests.\\
%
%\noindent {\bf Data availability.} The authors declare that their manuscript has no associated data.

\end{document}